\documentclass{article}
\usepackage[utf8]{inputenc}
\usepackage[english]{babel}
\usepackage{authblk}
\usepackage{graphicx}
\usepackage{epstopdf}
\usepackage[left=2.5cm,right=2.5cm,top=2.5cm,bottom=3cm,foot=1.5cm]{geometry}
\usepackage{amsmath,amssymb,latexsym, amsthm, bbm}

\newtheorem{theorem}{Theorem}
\newtheorem{lemma}{Lemma}
\newtheorem{proposition}{Proposition}

\begin{document}

\title{Cubic sublattices}
\author{M\'arton Horv\'ath}
\affil{Department of Geometry, Institute of Mathematics,\\
Budapest University of Technology and Economics,
Hungary\\\texttt{horvathm@math.bme.hu}}
\date{}
\maketitle
\let\thefootnote\relax\footnotetext{~\hspace{-2\parindent}\textit{2020 Mathematics Subject Classification:} Primary 52C07, Secondary 11H06.
}
\begin{abstract}
A sublattice of the three-dimensional integer lattice $\mathbb Z^3$ is called cubic sublattice if there exists a basis of the sublattice whose elements are pairwise orthogonal and of equal lengths.
We show that for an integer vector $\mathbf v\in\mathbb Z^3$ whose squared length is divisible by $d^2$, there exists a cubic sublattice containing $\mathbf v$ with edge length $d$. 
This improves one of the main result of a paper \cite{Goswick} of Goswick et al., where similar theorem was proved by using the decomposition theory of Hurwitz integral quaternions. We give an elementary proof heavily using cross product. This method allows us to characterize the cubic sublattices.
\end{abstract}

\section{Introduction}
Constructing lattice squares in the integer lattice  $\mathbb Z^2$ is easy as we can take a vector and its image under a $90$ degree rotation. The analogous problem in dimension three is to find lattice cubes in $\mathbb Z^3$. Although there are axis-parallel lattice cubes, still many other constructions exist, which are much more complicated to construct as there is no canonical rotation in dimension three.

One early result on lattice cubes is due to A.~S\'ark\" ozy \cite{Sarkozy}, who described some constructions of lattice cubes and determined the number of certain lattice cubes in 1961. This topic is still researched, E.~J.~Ionascu studied lattice cubes in dimensions 2, 3, 4 and determined their Ehrhart polynomial in his recently published paper \cite{Ionascu}.

If the edge length of a lattice cube is $d$, then its volume $d^3$ is the determinant of the edge vectors, so it is an integer, while the squared edge length $d^2$ is also an integer. This means that the edge length $d$ is an integer.
The following statement is an easy consequence of the description of the Pythagorean quadruples (see \cite{Carmichael} or \cite{Spira}).
If the length of a vector $\mathbf v\in\mathbb Z^3$ is an integer, then $\mathbf v$ can be extended to a lattice cube. An algorithmic proof is found in \cite{Parris}.
Similar questions are discussed in dimension four using Hurwitz integral quaternions by E.~W.~Kiss and P.~Kutas \cite{Kutas}.

A lattice cube can be easily extended to a cubic sublattice, which is a sublattice having such basis whose elements are pairwise orthogonal and of equal lengths. In this situation, it is naturally to ask which cubic sublattices contain a given vector $\mathbf v\in\mathbb Z^3$ not necessarily as a basis vector. If the edge length of the cubic sublattice is $d$, then $d^2$ must divide the squared length of $\mathbf v$. Our goal is to show that this condition is sufficient in the following sense.
\begin{theorem}\label{main}
For a vector $\mathbf v\in\mathbb Z^3$ whose squared length is divisible by $d^2$ for an integer $d$, there exists a cubic sublattice containing $\mathbf v$ with edge length $d$. If $\mathbf v$ is primitive, then this cubic sublattice is unique.
\end{theorem}
For example, the squared length of the vector $\mathbf v=(5,5,2)$ is $54$, which is divisible by $9$, so there exists a cubic sublattice containing $\mathbf v$ with edge length $3$, see Figure~\ref{fig:mainthm}.
\begin{figure}[h]
\centering
\includegraphics[width=12.5cm]{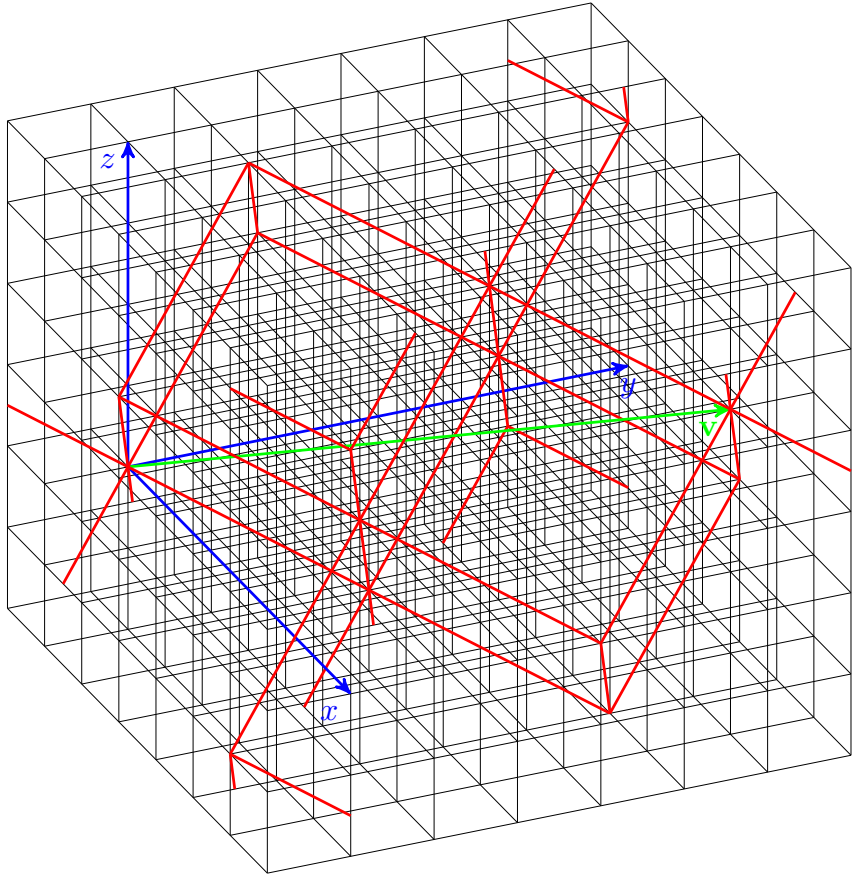}
\caption{The vector $\mathbf v=(5,5,2)$ is contained in the cubic sublattice with edge length $3$ generated by vectors $(-1,2,2), (2,-1,2), (2,2,-1)$. The figure shows the domain $[-1,6]\times[-1,6]\times[-2,4]$.}
\label{fig:mainthm}
\end{figure}

There might be several suitable cubic sublattices if $\mathbf v$ is not primitive. For instance, the vector $\mathbf v=(5,0,0)$ is contained in the cubic sublattices with edge length $d=5$ generated by bases $\{(5,0,0),(0,5,0),(0,0,5)\}$ and $\{(5,0,0),(0,3,4),(0,4,-3)\}$.

For the uniqueness, it is enough to assume that the greatest positive divisor $k$ of $\mathbf v$ and $d$ are coprime.
Indeed, we show that the primitive vector $\mathbf u=\mathbf v/k$ is also contained in the cubic sublattice $\Gamma$ given by the theorem for $\mathbf v$ when $d$ and $k$ are coprime. In the group $\mathbb Z^3/\Gamma$, the order of $\mathbf u\Gamma$ divides $k$ and the order $d^3$ of the group by Lagrange's theorem, hence the order of $\mathbf u\Gamma$ is $1$.

For the greatest possible value of $d$, i.e., in the case when $\|\mathbf v\|^2/d^2$ is square-free, Theorem~\ref{main} was proved in \cite{Goswick} by L.~M.~Goswick, E.~W.~Kiss, G.~Moussong and N.~Sim\'anyi using the decomposition theory of Hurwitz integral quaternions. Our proof builds solely on the structure of the three-dimensional Euclidean space and some basic facts about lattice geometry.

Finally, we give a number-theoretic corollary by considering the squared length of the vector $\mathbf v$. When a number $d^2m$ (where $d$ and $ m$ are integers) is a sum of three squares, then $m$ is also a sum of three squares. This is not surprising because Legendre's three-square theorem states that a natural number is a sum of three squares if and only if it is not of the form $4^n(8k+7)$. As a primitive vector remains primitive vector in the cubic sublattice, we can formulate a similar corollary: If an integer $d^2m$ is a sum of three coprime squares, then $m$ is also a sum of three coprime squares. We will discuss the converse of this claim after Theorem~\ref{thm:reverse} at the end of Section~\ref{characterization}.

In Section~\ref{preliminaries}, we introduce some definitions and propositions from lattice geometry. Section~\ref{proof} is devoted to the proof of Theorem~\ref{main}. Then we characterize the cubic sublattices of $\mathbb Z^3$ and we prove a kind of reverse theorem in Section~\ref{characterization}. 
Finally, we investigate cubic sublattices as a partial ordered set in Section~\ref{poset}.

\section{Preliminaries}\label{preliminaries}
In this section, we summarize some basic definitions and statements without proof. We refer to \cite{Cassels} for more details on lattice geometry. After that we prove some easy propositions which will be used later.

Let $\Lambda$ be an $n$-dimensional lattice in a real vector space. We say that a subgroup $K<\Lambda$ is a sublattice if $K$ is also an $n$-dimensional lattice. This is equivalent to the finiteness of the index of $K$ in $\Lambda$. If some linearly independent vectors in $\Lambda$ form a basis in the intersection of $\Lambda$ and the linear subspace generated by them, then these vectors can be extended to a basis of $\Lambda$. The parallelepiped generated by the basis of the lattice is called fundamental parallelepiped. When a scalar product or just a volume form is given on the vector space, the volume of the fundamental parallelepiped does not depend on the choice of the basis of the lattice. For a sublattice $K\subseteq\Lambda$, the ratio of the volumes of the fundamental parallelepipeds is equal to the index of $K$ in $\Lambda$ as a subgroup.

We say that a vector $\mathbf v\in \Lambda$ is divisible by a positive integer $k$ if there exists a vector $\mathbf u\in \Lambda$ such that $\mathbf v=k\mathbf u$. A vector $\mathbf v\in \Lambda$ is called primitive if there is not any positive integer $k\neq1$ which divides $\mathbf v$. For a non-zero vector $\mathbf v\in \Lambda$, there exists a unique primitive vector $\mathbf u$ and a unique positive integer $k$ such that $\mathbf v=k\mathbf u$. In this case, we say that $k$ is the greatest divisor of $\mathbf v$. 
The vectors that are divisible by a positive integer $k$ form a sublattice in $\Lambda$, it will be denoted by $k\Lambda$. The index of $k\Lambda$ in $\Lambda$ is $k^n$.

Fixing a basis $\{\mathbf e^1,\dots,\mathbf e^n\}$ of $\Lambda$, we can identify $\Lambda$ with $\mathbb Z^n$ by using the coordinates $(v_1,\dots,v_n)$ of a vector $\mathbf v\in\Lambda$ with respect to this basis. 
A vector is divisible by $k$ if and only if its coordinates so are. A vector is primitive if and only if its coordinates are coprime. The standard embedding $\mathbb Z^n\subseteq \mathbb R^n$ and the Euclidean vector space structure of $\mathbb R^n$ define the dot product on $\Lambda$. In particular, the perpendicularity of two vectors of $\Lambda$ and the length $\|\mathbf v\|$ of a vector $\mathbf v\in\Lambda$ are defined. In this case, the volume of the fundamental parallelepiped is the determinant of the basis vectors.

From now, we consider the case $n=3$. The advantage of this dimension is the applicability of the cross product. 
We say that a sublattice $\Gamma$ of the standard lattice $\mathbb Z^3$ is a cubic sublattice if there exists a basis of $\Gamma$ whose elements are pairwise orthogonal and of equal lengths. Such basis is called cubic basis, and their common length $d$ is the edge length. When we have a cubic sublattice $\Gamma\subset\mathbb Z^3$, we can identify $\Gamma$ with $\mathbb Z^3$, so we can measure the vectors of $\Gamma$ with respect to this identifying.

For a non-zero vector $\mathbf v\in\mathbb Z^3$, the set of the orthogonal vectors to $\mathbf v$ will be denoted by $\mathbf v^\perp$.

\begin{proposition}\label{perp-lattice}
For a non-zero vector $\mathbf v=(v_1,v_2,v_3)\in\mathbb Z^3$, the set $\mathbf v^\perp$ is a two-dimensional lattice.
\end{proposition}
\begin{proof}
Obviously, $\mathbf v^\perp$ is a discrete subgroup. Its dimension is at most $2$, we show that it is exactly $2$. We can assume that $v_3\neq0$, so none of the linear combinations of $\mathbf e^1,\mathbf e^2$ is parallel to $\mathbf v$. In this case, the linearity of the cross product yields that the vectors $\mathbf e^1\times\mathbf v,\mathbf e^2\times\mathbf v\in\mathbf v^\perp$ are linearly independent.
\end{proof}

\begin{proposition}\label{perp_area}
For a primitive vector $\mathbf v\in\mathbb Z^3$, the area of the fundamental parallelogram in $\mathbf v^\perp$ is equal to the length of $\mathbf v$.
\end{proposition}
\begin{proof}
The area of the  fundamental parallelogram equals the length of the cross product $\tilde{\mathbf v}$ of the generating vectors. As $\mathbf v$ is primitive, $\tilde{\mathbf v}$ is a multiple of $\mathbf v$.
Every pair of the vectors $\mathbf e^i\times\mathbf v$ ($i=1,2,3$) generates a sublattice of $\mathbf v^\perp$. The cross product of the generators is
\[(\mathbf e^i\times\mathbf v)\times(\mathbf e^j\times\mathbf v)=\left(\mathbf v\cdot(\mathbf e^i\times\mathbf e^j)\right)\mathbf v=\pm v_k\mathbf v,\]
where $i,j,k\in\{1,2,3\}$ are distinct indices. These vectors are multiples of $\tilde{\mathbf v}$. As $\mathbf v$ is primitive, the coefficients $v_1,v_2,v_3$ are coprime, hence $\tilde{\mathbf v}=\pm\mathbf v$.
\end{proof}

Fixing a primitive vector $\mathbf v$, let the vectors $\mathbf a,\mathbf b\in\mathbb Z^3$ be equivalent if their difference is a multiple of $\mathbf v$. Then the cross product with $\mathbf v$ is a well-defined map $\Phi_{\mathbf v}$ from the equivalence classes to $\mathbf v^\perp$.
\begin{proposition}\label{equivalent}
The map $\Phi_{\mathbf v}$ is a bijection.
\end{proposition}
\begin{proof}
If the cross products with $\mathbf v$ are equal for two vectors of $\mathbb Z^3$, then their difference is parallel to $\mathbf v$. This means that $\Phi_{\mathbf v}$ is injective.

For the surjectivity, consider a vector $k\mathbf u\in\mathbf v^\perp$, where $\mathbf u$ is primitive and $k$ is an integer. The vector $\mathbf v$ is primitive in $\mathbf u^\perp$, so there exists a vector $\mathbf w$ such that $\{\mathbf v,\mathbf w\}$ is a basis of $\mathbf u^\perp$. Then the cross product $\mathbf w\times \mathbf v$ is $\pm\mathbf u$, so $\pm k\mathbf w\times \mathbf v=k\mathbf u$, hence $\Phi_{\mathbf v}$ is surjective.
\end{proof}

\section{The proof of Theorem~\ref{main}}\label{proof}
 
It is enough to prove the theorem for primitive vector $\mathbf v$.
Indeed, if $\mathbf v=k\mathbf u$ for a primitive vector $\mathbf u\in\mathbb Z^3$, and $d^2$ divides the squared length $\|\mathbf v\|^2=k^2\|\mathbf u\|^2$, then there exists a decomposition $d=d_1d_2$ such that $d_1$ divides $k$ and $d_2^2$ divides $\|\mathbf u\|^2$. Applying the theorem for $d_2$ and $\mathbf u$, we get a cubic sublattice $\Gamma$. Then the cubic sublattice $d_1\Gamma$ has edge length $d_1d_2=d$ and contains $d_1\mathbf u$, therefore also $\mathbf v$.

First we prove the uniqueness part of the theorem in case $\mathbf v$ is primitive. Suppose that we have a cubic sublattice $\Gamma$ containing $\mathbf v$ with edge length $d$.
The first lemma is the key observation.
\begin{lemma}\label{lem:basic}
If $\mathbf a,\mathbf b\in \Gamma$, then $\mathbf a\times\mathbf b$ is divisible by $d$, and $\mathbf a\times\mathbf b/d\in \Gamma$.
\end{lemma}
\begin{proof}
Computing the cross product of $\mathbf a$ and $\mathbf b$ with respect to the Euclidean structure of $\Gamma\cong\mathbb Z^3\subset\mathbb R^3$, we get $\mathbf a\times\mathbf b/d$, which implies the statement.
\end{proof}
Lemma~\ref{lem:basic} yields that $\mathbf a\times\mathbf v$ is divisible by $d$ for $\mathbf a\in\Gamma$, so consider the following subset of $\mathbf v^\perp$
\[M(\mathbf v,d)=\{\mathbf a\in\mathbf v^\perp\mid \mathbf a\times\mathbf v \text{ is divisible by }d\}.\]
We have that $M(\mathbf v,d)$ contains the intersection $\mathbf v^\perp\cap \Gamma$. Later we will see that these sets are coincide.
By the linearity of the cross product, $M(\mathbf v,d)$ is a subgroup. The next lemma shows that it is a sublattice.

\begin{lemma}
The index of $M(\mathbf v,d)$ in $\mathbf v^\perp$ is $d$.
\end{lemma}
\begin{proof}
Let $i$ denote the index of $M(\mathbf v,d)$ in $\mathbf v^\perp$.
Firstly we prove that $i$ is divisible by $d$.
As the vector $\mathbf v=(v_1,v_2,v_3)$ is primitive, there exist integers $t_1,t_2,t_3$ such that $t_1v_1+t_2v_2+t_3v_3=1$. Consider the vector $\mathbf t=(t_1,t_2,t_3)$, and put $\mathbf u=\mathbf t\times\mathbf v\in\mathbf v^\perp$. Then
\[\mathbf u\times\mathbf v=(\mathbf t\times\mathbf v)\times\mathbf v=(\mathbf t\cdot\mathbf v)\mathbf v-(\mathbf v\cdot\mathbf v)\mathbf t=\mathbf v-\ell^2\mathbf t=\left(v_1-\ell^2t_1, v_2-\ell^2t_2,v_3-\ell^2t_3\right).\]
This vector is not divisible by any non-unit divisor of $d$, otherwise such a divisor would divide $\ell^2$ and $v_1,v_2,v_3$, which would contradict the primitiveness of $\mathbf v$.
Thus, the vectors $\mathbf u\times\mathbf v, 2\mathbf u\times\mathbf v,\dots, (d-1)\mathbf u\times\mathbf v$ are not divisible by $d$, so the vectors $\mathbf u, 2\mathbf u,\dots, (d-1)\mathbf u$ are not contained in $M(\mathbf v,d)$. The vector $d\mathbf u$ belongs to $M(\mathbf v,d)$ by definition. In the group $\langle\mathbf u\rangle$ generated by $\mathbf u\in\mathbf v^\perp$, the subgroup $\langle\mathbf u\rangle\cap M(\mathbf v,d)$ has index $d$, hence $i$ is a multiple of $d$ by Noether's isomorphism theorem.

Now we prove that $i$ divides $d$.
Consider the vectors
\begin{align*}
\mathbf r^1&=d\mathbf e^1\times\mathbf v=(0,-dv_3,dv_2),\\
\mathbf s^1&=(\mathbf r^1\times\mathbf v)/d=(\mathbf e^1\times\mathbf v)\times\mathbf v=\left(-v_2^2-v_3^2,v_1v_2,v_1v_3\right).
\end{align*}
 We have that $\mathbf r^1\in M(\mathbf v,d)$ and $\mathbf r^1\perp\mathbf s^1$. As $\mathbf r^1$ is perpendicular to $\mathbf v$, we have 
\[\mathbf s^1\times\mathbf v=\frac1d(\mathbf r^1\times\mathbf v)\times\mathbf v=-\frac{\ell^2}d\mathbf r^1,\]
which means that $\mathbf s^1\in M(\mathbf v,d)$. The area of the rectangle spanned by $\mathbf r^1$ and $\mathbf s^1$ is $\|\mathbf r^1\|^2\ell/d=d\ell(v_2^2+v_3^2)$. By Proposition~\ref{perp_area}, the area of the fundamental parallelogram in $\mathbf v^\perp$ is $\ell$.
The index of the sublattice generated by $\mathbf r^1$ and $\mathbf s^1$ in $\mathbf v^\perp$ is the quotient of the areas of the fundamental parallelograms, which is equal to $d(v_2^2+v_3^2)$.
This shows that $i$ divides $d(v_2^2+v_3^2)$. We get similarly that $i$ also divides $d(v_1^2+v_3^2)$ and $d(v_1^2+v_2^2)$, which yields that $i$ divides the greatest common divisor
\begin{align*}
\gcd\left(d(v_2^2+v_3^2),d(v_1^2+v_3^2),d(v_1^2+v_2^2)\right)
&=d\gcd\left(v_2^2+v_3^2,v_1^2+v_3^2,v_1^2+v_2^2\right)\\
{}&=d\gcd\left(v_2^2+v_3^2,v_1^2+v_3^2,v_1^2+v_2^2,v_2^2-v_3^2,v_1^2-v_3^2,v_1^2-v_2^2\right)\\
{}&=d\gcd\left(2v_1^2,2v_2^2,2v_3^2,v_2^2+v_3^2,v_1^2+v_3^2,v_1^2+v_2^2\right).
\end{align*}
Using $\gcd(v_1,v_2,v_3)=1$, we obtain $\gcd(d(v_2^2+v_3^2),d(v_1^2+v_3^2),d(v_1^2+v_2^2))=d$ if $v_1,v_2,v_3$ are not all odd. This implies the statement in this case. 
If $v_1,v_2,v_3$ are all odd, then $\gcd(d(v_2^2+v_3^2),d(v_1^2+v_3^2),d(v_1^2+v_2^2))=2d$, so we have that $i$ divides $2d$. In this case, $\ell^2=v_1^2+v_2^2+v_3^2$ is odd, so is $d$. The sublattice $d\mathbf v^\perp$ is contained in $M(\mathbf v,d)$, and its index is $d^2$. Therefore $i$ divides $\gcd(2d,d^2)=d$.
\end{proof}
Lemma~\ref{lem:basic} implies for a vector $\mathbf a\in \Gamma$ that $\mathbf a\times\mathbf v$ is divisible by $d$ and $\mathbf a\times\mathbf v/d\in \Gamma$. Since $\mathbf a\times\mathbf v/d\in\mathbf v^\perp$ as well, we obtain $\mathbf a\times\mathbf v/d\in M(\mathbf v,d)$. This motivates the definition of the set
\begin{align*}
\Gamma(\mathbf v,d)&=\{\mathbf a\in\mathbb Z^3\mid\mathbf a\times\mathbf v/d\in M(\mathbf v,d) \}\\
&=\{\mathbf a\in\mathbb Z^3\mid\mathbf a\times\mathbf v\text{ is divisible by $d$, and }(\mathbf a\times\mathbf v)\times\mathbf v\text{ is divisible by $d^2$}\}.
\end{align*}
We have $\Gamma\subseteq \Gamma(\mathbf v,d)$ and $\mathbf v\in \Gamma(\mathbf v,d)$. As the cross product is linear, $\Gamma(\mathbf v,d)$ is a subgroup. The following lemma shows that $\Gamma(\mathbf v,d)$ is a sublattice.
\begin{lemma}
The index of $\Gamma(\mathbf v,d)$ in $\mathbb Z^3$ is $d^3$.
\end{lemma}
\begin{proof}
The sublattice $\Gamma(\mathbf v,d)$ contains such vectors $\mathbf a\in\mathbb Z^3$ that $\mathbf a\times\mathbf v\in dM(\mathbf v,d)$. 
 By Proposition~\ref{equivalent}, the index of $\Gamma(\mathbf v,d)$ in $\mathbb Z^3$ is equal to the index of $dM(\mathbf v,d)$ in $\mathbf v^\perp$.
As $M(\mathbf v,d)$ has index $d$ in $\mathbf v^\perp$, the index of the sublattice $dM(\mathbf v,d)$ in $\mathbf v^\perp$ is $d^3$.
\end{proof}
Since the index of $\Gamma$ is $d^3$ as well, if there exists an appropriate cubic sublattice, then it is $\Gamma(\mathbf v,d)$, which proves the uniqueness part of Theorem~\ref{main}.

\vspace{3mm}
Now we prove that $\Gamma(\mathbf v,d)$ is indeed a cubic sublattice. The first step is to show that the dot products of the elements of $\Gamma(\mathbf v,d)$ are divisible by $d^2$, but we have to start with weaker lemmas.
\begin{lemma}\label{p1}
For a vector $\mathbf a\in \Gamma(\mathbf v,d)$, the dot product $\mathbf a\cdot\mathbf v$ is an integer and it is divisible by $d^2$.
\end{lemma}
\begin{proof}
For $\mathbf a\in \Gamma(\mathbf v,d)$, we have that $d^2$ divides
\[(\mathbf a\times\mathbf v)\times\mathbf v=(\mathbf a\cdot\mathbf v)\mathbf v-(\mathbf v\cdot\mathbf v)\mathbf a,\]
which means that
\[\frac{\mathbf a\cdot\mathbf v}{d^2}\mathbf v-\frac{\mathbf v\cdot\mathbf v}{d^2}\mathbf a\in\mathbb Z^3.\]
Since $\mathbf v\cdot\mathbf v$ is divisible by $d^2$, we earn $\dfrac{\mathbf a\cdot\mathbf v}{d^2}\mathbf v\in\mathbb Z^3$. This gives the statement as $\mathbf v$ is primitive.
\end{proof}

\begin{lemma}\label{p2}
For vectors $\mathbf a,\mathbf b\in \Gamma(\mathbf v,d)$, the vector $\dfrac{\mathbf a\times\mathbf b}{d}$ is contained in $\Gamma(\mathbf v,d)$.
\end{lemma}
\begin{proof}
Our goal is to prove that $M(\mathbf v,d)$ contains the vector
\[\frac{\frac{\mathbf a\times\mathbf b}{d}\times\mathbf v}{d}=\frac{\mathbf a\cdot\mathbf v}{d^2}\mathbf b-\frac{\mathbf b\cdot\mathbf v}{d^2}\mathbf a,\]
where the coefficients $\dfrac{\mathbf a\cdot\mathbf v}{d^2}$ and $\dfrac{\mathbf b\cdot\mathbf v}{d^2}$ are integers by Lemma~\ref{p1}.
The above vector is in $M(\mathbf v,d)$ if
\[\frac{\mathbf a\cdot\mathbf v}{d^2}\frac{\mathbf b\times\mathbf v}{d}-\frac{\mathbf b\cdot\mathbf v}{d^2}\frac{\mathbf a\times\mathbf v}d\in\mathbb Z^3,\]
which is satisfied, because $\dfrac{\mathbf b\times\mathbf v}{d},\dfrac{\mathbf a\times\mathbf v}d\in\mathbb Z^3$ as $\mathbf a,\mathbf b\in \Gamma(\mathbf v,d)$.
\end{proof}

\begin{lemma}\label{p3}
For vectors $\mathbf a,\mathbf b\in \Gamma(\mathbf v,d)$, the dot product $\mathbf a\cdot\mathbf b$ is divisible by $d^2$.
\end{lemma}
\begin{proof}
We have $\dfrac{\mathbf a\times\mathbf v}{d}\in\Gamma(\mathbf v,d)$ by Lemma~\ref{p2}. Applying Lemma~\ref{p2} again, we get that the vector 
\[(\mathbf a\times\mathbf v)\times\mathbf b=(\mathbf a\cdot\mathbf b)\mathbf v-(\mathbf v\cdot\mathbf b)\mathbf a\]
is divisible by $d^2$. Lemma~\ref{p1} implies that $d^2$ divides the coefficient $\mathbf v\cdot\mathbf b$. Therefore the vector $(\mathbf a\cdot\mathbf b)\mathbf v$ is also divisible by $d^2$, which gives the statement.
\end{proof}
\pagebreak
The second step is to find the cubic basis in $\Gamma(\mathbf v,d)$.
\begin{lemma}\label{p11}
There exists a vector $\mathbf a\in \Gamma(\mathbf v,d)$ of length $d$.
\end{lemma}
\begin{proof}
Suppose the contrary. Lemma~\ref{p3} implies that the squared length of every vector of $\Gamma(\mathbf v,d)$ is divisible by $d^2$. This and the indirect assumption yield that the length of every non-zero vector of $\Gamma(\mathbf v,d)$ is at least $\sqrt2d$. Thus, the balls centered at the elements of $\Gamma(\mathbf v,d)$ with radius $\sqrt2d/2$ are disjoint. The parts of the intersection of a fundamental parallelepiped and these balls form an entire ball. Its volume is
\[\frac{4\pi}3\left(\frac{\sqrt2d}{2}\right)^3>\sqrt2d^3,\]
while the volume of the fundamental parallelepiped is $d^3$. This is a contradiction.
\end{proof}
Fix a  vector $\mathbf a\in \Gamma(\mathbf v,d)$ of length $d$. By Lemma~\ref{p3}, for a vector $\mathbf b\in \Gamma(\mathbf v,d)$, the dot product $\mathbf a\cdot\mathbf b$ is divisible by $d^2$, so the length of the projection of $\mathbf b$ to $\mathbf a$ is a multiple of $d$. Therefore the elements of $\Gamma(\mathbf v,d)$ lie in $\mathbf a^\perp$ and its translates by the multiples of $\mathbf a$. Consider the sublattice $\Lambda=\mathbf a^\perp\cap \Gamma(\mathbf v,d)$. The area of the fundamental parallelogram in $\Lambda$ is $d^2$ as we can get the basis of the fundamental parallelepiped of $\Gamma(\mathbf v,d)$ by adding $\mathbf a$ to the basis of the fundamental parallelogram of $\Lambda$.
\begin{lemma}\label{p12}
There exists a vector $\mathbf b\in \Lambda$ of length $d$.
\end{lemma}
\begin{proof}
We prove by contradiction in a way similar to the proof of Lemma~\ref{p11}. We suppose that the length of every non-zero vector in $\Lambda$ is at least $\sqrt2d$. In this case, the disks around the elements of $\Lambda$ with radius $\sqrt2d/2$ in the plane perpendicular to $\mathbf a$ are disjoint. The area of the intersection of a fundamental parallelogram and these disks is
\[\left(\frac{\sqrt2d}{2}\right)^2\pi=\frac\pi2d^2>d^2,\]
which is the area of the fundamental parallelogram, this is a contradiction.
\end{proof}
Fix a vector $\mathbf b\in \Lambda$ of length $d$.
The vector $\mathbf c=\frac{\mathbf a\times\mathbf b}d$ is in $\Gamma(\mathbf v,d)$ by Lemma~\ref{p2}. Since the lengths of $\mathbf a$ and $\mathbf b$ are $d$ and they are perpendicular, the length of $\mathbf c$ is also $d$. The vectors $\mathbf a,\mathbf b,\mathbf c$ are pairwise perpendicular, so the volume of the parallelepiped generated by them is $d^3$. As the index of $\Gamma(\mathbf v,d)$ in $\mathbb Z^3$ is also $d^3$, the vectors $\mathbf a,\mathbf b,\mathbf c$ generate the sublattice $\Gamma(\mathbf v,d)$. This means that $\Gamma(\mathbf v,d)$ is indeed a cubic sublattice with edge length $d$.

\section{Characterization of cubic sublattices}\label{characterization}
Although the following two lemmas are well-known, we prove them for the sake of completeness.
\begin{lemma}\label{gcd2}
Let $K$ be a sublattice of a two-dimensional lattice $\Lambda$ and let $k$ be the greatest common divisor of the vectors of $K$. Then there exists a vector in $K$ whose greatest divisor is $k$.
\end{lemma}
\begin{proof}
Let the generators of $K$ be $\mathbf v_1=k_1\mathbf u_1$ and $\mathbf v_2=k_2\mathbf u_2$, where $\mathbf u_1,\mathbf u_2$ are primitive, and $k_1,k_2$ are positive integers. We can assume that $k=\gcd(k_1,k_2)$ is equal to $1$. Unfortunately, the vector $\mathbf v_1+\mathbf v_2$ is not necessarily primitive, see Figure \ref{fig:2dim}, so we have to be more tricky.
\begin{figure}[h]
\centering
\includegraphics[width=10cm]{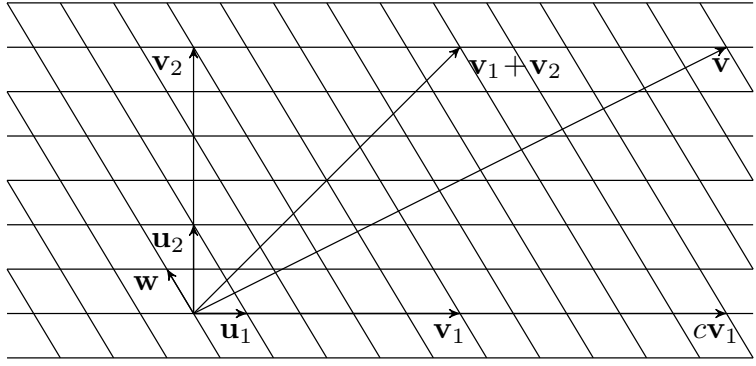}
\caption{The vector $\mathbf v_1+\mathbf v_2$ is divisible by $2$.}
\label{fig:2dim}
\end{figure}
As $\mathbf u_1$ is primitive, there exists such a vector $\mathbf w$ that $\mathbf u_1$ and $\mathbf w$ generate $\Lambda$. We can write $\mathbf u_2=a\mathbf u_1+b\mathbf w$, where $a,b$ are coprime integers. Let $c$ be the product of those prime numbers which divide $b$, but do not divide $k_1k_2$ (set $c=1$ if there is not any such prime number). We show that the vector $\mathbf v=c\mathbf v_1+\mathbf v_2\in K$ is a primitive vector. Suppose the contrary, i.e., there exists a prime number $p$ which divides 
\[\mathbf v=c\mathbf v_1+\mathbf v_2=ck_1\mathbf u_1+k_2(a\mathbf u_1+b\mathbf w)=(ck_1+ak_2)\mathbf u_1+bk_2\mathbf w.\]
Since $\mathbf u_1,\mathbf w$ are the generators of $\Lambda$, $p$ divides $ck_1+ak_2$ and $bk_2$. As $p$ is prime, $p$ divides $b$ or $k_2$. If $p$ divides $k_2$, then $p$ also divides $ck_1$, which contradicts the definition of $c$ or $\gcd(k_1,k_2)=1$. If $p$ does not divide $k_2$, then $p$ divides $b$. By the definition of $c$, $p$ divides $ck_1$, hence $p$ divides $ak_2$ and consequently also $a$. Thus, $p$ divides $\mathbf u_2=a\mathbf u_1+b\mathbf w$, which is a contradiction.
\end{proof}
The next lemma is the three-dimensional version of Lemma~\ref{gcd2}.
\begin{lemma}\label{gcd3}
If the greatest common divisor of the vectors of a sublattice $K$ of a three-dimensional lattice $\Lambda$ is $k$, then $K$ contains a vector whose greatest divisor is $k$.
\end{lemma}
\begin{proof}
Denote the generators of $K$ by $\mathbf v_1,\mathbf v_2,\mathbf v_3$ and its greatest divisors by $k_1,k_2,k_3$, respectively. We have that $k=\gcd(k_1,k_2,k_3)$. Let $M$ be the sublattice generated by $\mathbf v_1$ and $\mathbf v_2$. By Lemma~\ref{gcd2}, there is a vector $\mathbf v$ in $M$ whose greatest divisor is $\gcd(k_1,k_2)$. If we apply Lemma~\ref{gcd2} again to the sublattice generated by $\mathbf v$ and $\mathbf v_3$, we get a vector in $K$, whose greatest divisor is $\gcd(\gcd(k_1,k_2),k_3)=k$.
\end{proof}
We remark that similar statement holds for arbitrary dimensional lattices, and one can prove it by induction on the dimension.

Now we can characterize the cubic sublattices of $\mathbb Z^3$. We show that every cubic sublattice can be obtained from our construction.

\begin{theorem}\label{cubic_constr}
For an arbitrary cubic sublattice $\Gamma\subseteq\mathbb Z^3$, there exist unique positive integers $k$ and $d$ such that $\Gamma=k\Gamma(\mathbf v,d)$ for a primitive vector $\mathbf v\in\mathbb Z^3$.
\end{theorem}
\begin{proof}
We have that $k$ is the greatest common divisor of the vectors of $\Gamma$ and $d$ is the quotient of the edge length of $\Gamma$ and $k$. Consider the cubic sublattice $\Gamma'$ such that $k\Gamma'=\Gamma$.
By Lemma~\ref{gcd3}, there exists a primitive vector $\mathbf v\in\Gamma'$. Theorem~\ref{main} implies $\Gamma'=\Gamma(\mathbf v,d)$, which gives the statement.
\end{proof}
When a cubic sublattice is constructed from a primitive vector, this vector is also a primitive vector in the cubic sublattice. Our next goal is to show that every primitive vector can be such a vector. It requires two lemmas.
\begin{lemma}\label{d^2}
If a cubic sublattice $\Gamma\subseteq\mathbb Z^3$ has edge length $d$, then it contains the vectors divisible by $d^2$, i.e., the inclusion $d^2\mathbb Z^3\subseteq\Gamma$ is fulfilled.
\end{lemma}
\begin{proof}
When $\Gamma=\Gamma(\mathbf v,d)$ for some primitive vector $\mathbf v$, we have $d^2\mathbf a\in\Gamma(\mathbf v,d)$ for any $\mathbf a\in\mathbb Z^3$ by the construction.
In the general case, $\Gamma=k\Gamma(\mathbf v,d/k)$ for some integer $k$ by Theorem~\ref{cubic_constr}. For an arbitrary $\mathbf a\in\mathbb Z^3$, we obtain $d^2/k^2\mathbf a\in\Gamma(\mathbf v,d/k)$, hence $d^2\mathbf a\in \Gamma$.
\end{proof}
The second lemma is a claim from number theory.
\begin{lemma}\label{prime}
For an odd prime number $p$, there exists a primitive vector in $\mathbb Z^3$ whose squared length is divisible by $p^2$.
\end{lemma}
\begin{proof}
It is enough to construct a vector whose squared length is divisible by $p^2$ and which is not divisible by $p$, as we can divide it by its greatest divisor.

If $-1$ is a quadratic residue (when $p=4k+1$ for a positive integer $k$), then there exists a positive integer $x$ such that $x^2=ap-1$ for some integer $a$. As $p$ is odd, there exists integer $b$ such that $-2b\equiv a$ modulo $p$. Then $(0,x,bp+1)$ is a suitable vector.

When $-1$ is not a quadratic residue (if $p=4k+3$), we search for $x,y\in\mathbb Z_p$ such that $x^2+y^2=-1$. If there were no such elements, then each of the pairs 
\[(1,p-2),(2,p-3),\dots,((p-1)/2,(p-1)/2)\]
 would contain only one quadratic residue (in particular the last one would not contain any), which would contradict the well-known fact that there exist $(p-1)/2$ quadratic residues modulo $p$. Denoting the corresponding integers by $x,y$ as well, we obtain that $x^2+y^2=ap-1$ for some integer $a$. As in the previous case, we have an integer $b$ such that $-2b\equiv a$ modulo $p$, thus, the vector $(x,y,bp+1)$ is a suitable vector.
\end{proof}
Lemma~\ref{prime} does not hold for $p=2$. Indeed, the square numbers are congruent $0$ or $1$ modulo $4$, so if $4$ divides the sum of the square of the coordinates, then all of them are even, hence the vector is divisible by $2$.

If we construct a cubic sublattice in a cubic sublattice $\Lambda$ instead of $\mathbb Z^3$, we will use the notion $\Gamma_\Lambda(\mathbf v,d)$ for the cubic sublattice of $\Lambda$ with edge length $d$ which contains the primitive vector $\mathbf v\in\Lambda$.
\begin{theorem}\label{thm:reverse}
For an arbitrary primitive vector $\mathbf v=(v_1,v_2,v_3)\in\mathbb Z^3$ and an odd $d$, there exists a primitive vector $\mathbf u\in\mathbb Z^3$ for which we can choose a cubic basis of the cubic sublattice $\Gamma(\mathbf u,d)$ so that the coordinates of $\mathbf u$ are $(v_1,v_2,v_3)$.
\end{theorem}
\begin{proof}
First we show that it is enough to prove the theorem for odd prime number $d$.
Applying the theorem for a vector $\mathbf v=(v_1,v_2,v_3)\in\mathbb Z^3$ and $d_1$ yields a primitive vector $\mathbf u\in\mathbb Z^3$. If we apply again the theorem for $\mathbf u$ and for $d_2$, we get a primitive vector $\mathbf t\in\mathbb Z^3$.
The sublattice $\Gamma_{\Gamma(\mathbf t,d_2)}(\mathbf u,d_1)$ is a cubic sublattice with edge length $d_1d_2$, where the coordinates of the vector $\mathbf t$ are $(v_1,v_2,v_3)$. The uniqueness part of Theorem~\ref{main} implies that this sublattice is $\Gamma(\mathbf t,d_1d_2)$. In the following, we suppose that $d$ is an odd prime number.

As $\mathbf v$ is primitive, $d$ does not divide at least one of its coordinates, let it be the first coordinate $v_1$.
Lemma~\ref{prime} provides a primitive vector $\mathbf w=(w_1,w_2,w_3)$ whose squared length is divisible by $d^2$. Reordering the coordinates, we can achieve that $d$ does not divide the first coordinate $w_1$. Set $\tilde{\mathbf w}=(-w_1,w_2,w_3)$. Now one of the dot products $\mathbf v\cdot\mathbf w$ or $\mathbf v\cdot\tilde{\mathbf w}$ is not divisible by $d$, otherwise $d$ would divide their difference $2v_1w_1$, which would contradict our assumptions on $v_1$ and $w_1$ as $d$ is odd. Assume that $\mathbf v\cdot\mathbf w$  is not divisible by $d$.

Finally, construct the cubic sublattice $\Gamma(\mathbf w,d)$. 
By Lemma~\ref{d^2}, it contains the vector $d^2\mathbf v$. Consider this vector as a vector of the cubic sublattice $\Gamma(\mathbf w,d)$, and denote it by $\mathbf u$.
This vector is primitive since $d\mathbf v$ is not contained in $\Gamma(\mathbf w,d)$ by Lemma~\ref{p1}.
The sublattice $d^2\mathbb Z^3$ is a cubic sublattice in $\Gamma(\mathbf w,d)$ with edge length $d$ and it contains $\mathbf u$. Thus, the cubic sublattice $\Gamma_{\Gamma(\mathbf w,d)}(\mathbf u,d)$ is $d^2\mathbb Z^3$. Choosing the basis $(d^2\mathbf e^1,d^2\mathbf e^2,d^2\mathbf e^3)$ of $d^2\mathbb Z^3$, we have that the coordinates of $\mathbf u$ are $(v_1,v_2,v_3)$.
\end{proof}
Theorem~\ref{thm:reverse} allows us to formulate the converse of the number-theoretic corollary mentioned in the Introduction. If an integer $m$ is a sum of three coprime squares, then for an odd $d$, the number $d^2m$ is also a sum of three coprime squares. The assumption on the oddness of $d$ is necessary, as a number divisible by $4$ cannot be a sum of three coprime squares by the investigation of the remainders modulo $4$.

\section{Cubic sublattices as partially ordered set}\label{poset}
The inclusion gives a partial order over the cubic sublattices of $\mathbb Z^3$. This makes the set of the cubic sublattices to a partially ordered set. We decide whether this set is a lattice in the algebraic sense.

Similarly to the cubic sublattices, we can investigate the square sublattices of $\mathbb Z^2$. A square sublattice can be characterized by the invariance of the rotation of $90^\circ$. Therefore the intersection and the union of square sublattices are also square sublattices. These are the supremum and the infimum of the given square sublattices, so the partially ordered set of the square sublattices of $\mathbb Z^2$ forms a lattice in algebraic sense.

We show that the partially ordered set of the cubic sublattices of $\mathbb Z^3$ is not a lattice in algebraic sense. The squared length of $\mathbf v=(1,2,2)$ is $9=3^2$, so there exists the cubic sublattice $\Gamma(\mathbf v,3)\leq\mathbb Z^3$. 
\begin{figure}[h]
\centering
\includegraphics{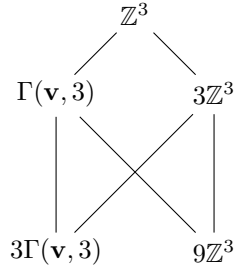}
\caption{This example shows that the partially ordered set of the cubic sublattices of $\mathbb Z^3$ does not form a lattice in algebraic sense.}
\label{fig:poset}
\end{figure}
We have the inclusions $9\mathbb Z^3\leq3\mathbb Z^3\leq\mathbb Z^3$ and $3\Gamma(\mathbf v,3)\leq\Gamma(\mathbf v,3)\leq\mathbb Z^3$, and from the last one, $3\Gamma(\mathbf v,3)\leq3\mathbb Z^3$, see Figure~\ref{fig:poset}. Lemma~\ref{d^2} gives $9\mathbb Z^3\leq\Gamma(\mathbf v,3)$. In all of the mentioned inclusions, the (relatively) edge lengths of the cubic sublattices are $3$, so there is not any cubic sublattice between the inclusions. This means that there is no supremum of $3\Gamma(\mathbf v,3)$ and $9\mathbb Z^3$. And there is no infimum of $\Gamma(\mathbf v,3)$ and $3\mathbb Z^3$.

However, the partially ordered set of such cubic sublattices that contain a given primitive vector $\mathbf v\in\mathbb Z^3$ is a lattice in algebraic sense. If the squared length of $\mathbf v$ is $kd^2$, where $k$ is square-free, then this lattice is isomorphic to the lattice of the divisors of $d$.

\section*{Acknowledgment}
The author is immensely grateful for G\'abor Moussong, who recommended this topic and proposed to search an elementary proof for Theorem~\ref{main}.

\end{document}